\documentclass{amsart}

\usepackage{amsfonts}       
\usepackage{amsthm}
\usepackage{amssymb}
\usepackage{graphicx}

\newtheorem{theorem}{Theorem}[section]
\newtheorem{lemma}[theorem]{Lemma}
\newtheorem{corollary}[theorem]{Corollary}

\theoremstyle{definition}

\theoremstyle{remark}
\newtheorem{remark}[theorem]{Remark}

\numberwithin{equation}{section}

\newcommand{\R}{\mathbb{R}}
\newcommand{\g}{\mathfrak{g}}
\newcommand{\grad}{\operatorname{grad}}

\begin{document}

\title{Hardy inequalities and uncertainty principles in the presence of a black hole}

\author{M. Paschalis}
\address{Department of Mathematics, National and Kapodistrian University of Athens}

\email{mpaschal@math.uoa.gr}


\date{\today}

\dedicatory{Preprint Version}


\begin{abstract}
In this paper we establish Hardy and Heisenberg uncertainty-type inequalities for the exterior of a Schwarzschild black hole. The weights that appear in both inequalities are tailored to fit the geometry, and can both be compared to the related Riemannian distance from the event horizon to yield inequalities for that distance. Moreover, in both cases the classic Euclidean inequalities with a point singularity can be recovered in the limit where one stands ``far enough'' from the black hole, as expected from the asymptotic flatness of the metric.
\end{abstract}

\maketitle

\tableofcontents

\section{Introduction}
One of the crowning achievements of Einstein's theory of general relativity was without doubt the prediction of spacetime singularities and black holes. The first category refers (quite obviously) to spacetime metrics that exhibit some kind of singularity, while the later one refers to the existence of an event horizon whose interior is causally inaccessible to external observers. The simplest, and most well known such example is the Schwarzschild metric
\begin{equation}
-\frac{r-1}{r} dt \otimes dt +\frac{r}{r-1} dr \otimes dr + r^2 g_{S^{2}},
\end{equation}
in $\R^{3+1}$ where $t\in \R$ and $r=|x|>0$. This metric is obtained by solving Einstein's equations in vacuum under the assumption of spherical symmetry, in which case a point mass $m$ exists in the origin ($2m=1$ in our simplified, dimensionless units). A straighforward generalisation is the $(n+1)$-dimensional analogue 
\begin{equation}
-\frac{r^{n-2}-1}{r^{n-2}} dt \otimes dt +\frac{r^{n-2}}{r^{n-2}-1} dr \otimes dr + r^2 g_{S^{n-1}},
\end{equation}
which is a solution to Einstein's equations in higher dimensions and possesses similar properties.

Black holes have been proven to be ubiquitous in the universe at the macroscopic level, recently even so by direct observation. However the existence of black holes of small mass is theoretically possible as well - these are called micro-black holes in the context of physics. This warrants the study of semi-classical, low energy quantum problems, i.e. the study of Schr\"odinger's equation with a background Schwarzschild geometry, as well as Heisenberg's uncertainty principle in that context. For more background on general relativity one may consult Carroll \cite{Carroll}, and for higher-dimensional black holes Horowitz \cite{Horowitz} and references therein (and of course this is just a selection from a vast literature).

The classic $L^2$ Hardy inequality for a point singularity on $\mathbb{R}^n$, $n\geq 3$, reads 
\begin{equation}
\int_{\mathbb{R}^n} |\nabla \psi(x)|^2 dx \geq \Big( \frac{n-2}{2} \Big)^2 \int_{\mathbb{R}^n} \frac{|\psi(x)|^2}{|x|^2} dx
\end{equation}
and holds true for all $\psi \in C^1_c(\mathbb{R}^n)$. It has important applications in the theory of Schr\"odinger operators involving singular potentials, and for this reason it is of great theoretical as well as applied interest, as it guarantees the non-negativity of the spectrum of the related Schr\"odinger operator 
\begin{equation}
-\Delta - \frac{\mu}{|x|^2}
\end{equation}
provided that $\mu\leq (n-2)^2/4$. The stability of the hydrogen atom is a well known consequence of this fact, to mention only one application. Note that the exponent $2$ in the denominator as well as the constant $(n-2)^2/4$ are both critical, which means that they cannot be improved (although this does not exclude the possibility of additional improving terms).

In the last few decades there has been a surge of research regarding the study of Hardy inequalities in different contexts. Another version of the Hardy inequality involves the distance $d$ from the boundary of a domain $\Omega$ rather than a single point. In particular, for convex domains it is well known that
\begin{equation}
\int_\Omega |\nabla \psi(x)|^2 dx \geq \frac{1}{4} \int_\Omega \frac{|\psi(x)|^2}{d^2(x)} dx,
\end{equation}
and again both the exponent $2$ and the constant $1/4$ are the best possible in the sense described above. For a general background on Hardy inequalities see for instance \cite{A,AM,BFT,D,H,MMP} and references therein, as well as the recent book \cite{BEL}.

There have also been developments regarding Hardy inequalities in various settings beyond the Euclidean, including the hyperbolic and spherical setting. Some notable examples are \cite{Ca,BGG,DD,KO,SP}, although the list is far from exhaustive.

In the present work, we aim to establish Hardy type inequalities in the context where the classical Euclidean metric is replaced by the Schwarzschild metric. This requires a careful analysis of the underlying geometry in order to determine a critical potential which is comparable with the relevant Riemannian distance function. In the same spirit we establish a Heisenberg uncertainty-type inequality with a similar, although different, weight function. Both results reduce to their Euclidean analogues when the distance from the event horizon is large, which is to be expected due to asymptotic flatness. It should be also noted that both results are restricted to the exterior of the black hole, as extending them to the interior would be essentially meaningless both physically and mathematically, due to considerations that will be outlined in the next section.

The Hardy inequality which we study in detail here was first derived by the author in \cite{P} in the case $n=3$, as a demonstration of the very general technique developed in that paper that allows to establish Hardy potentials in a wide geometric context, where it was left as an open problem to conduct a more thorough analysis on this specific case. Here we offer a different, analytic proof, in which we enlarge the class of admissible functions. The Heisenberg inequality is entirely novel.

\section{Schwarzschild geometry}

Since the essence of the problem lies in its geometry, it is important to first develop an adequate understanding of the geometry itself. We will concern ourselves only with Schwarzschild black holes, since they are the simplest both physically and mathematically.

The classic Schwarzschild spacetime featuring a black hole of radius $r_S=1$ in $(3+1)$-dimensional spacetime is defined as $\mathbb{R} \times \mathbb{R}^3\setminus\{0\}$, equipped with the Lorentzian metric 
\begin{align*}
-\frac{r-1}{r} dt \otimes dt +\frac{r}{r-1} dr \otimes dr + r^2 g_{S^{2}},
\end{align*}
as it is expressed when choosing polar coordinates for the spatial part. This is a direct solution of Einstein's equations in the case where no charge or angular momentum is assumed.

For the sake of depth and generality, we will actually take things a step further and consider the $(n+1)$-dimensional analogue of this which emerges as a solution of Einstein's equations in higher dimensions. Assuming spatial dimension $n \geq 3$, these generalised Schwarzschild spacetimes are likewise defined as $\mathbb{R} \times \mathbb{R}^n\setminus\{0\}$, carrying the metric 
\begin{align*}
-\frac{r^{n-2}-1}{r^{n-2}} dt \otimes dt +\frac{r^{n-2}}{r^{n-2}-1} dr \otimes dr + r^2 g_{S^{n-1}}.
\end{align*}
The metric seemingly contains two singularities, one for $r=0$ and one for $r=1$ (i.e. $\R\times S^{n-1}$). However, it is well known that only $r=0$ corresponds to a true singularity, while the case $r=1$ is a coordinate artifact that can be eliminated by passing to different coordinates. Nevertheless, is is true that a so called \textit{event horizon} is formed at $r=1$, which makes the interior causally inaccessible to the exterior, and this is one of the main reasons why our semi-classical model will have to be restricted to the exterior region $\R\times \mathcal{E}$ where $\mathcal{E}:= \{ x\in \R^n: |x|>1\}$. (In fact, the Schwarzschild metric is such that $dt$ corresponds to the proper time element of a stationary observer standing far enough from the black hole, and it is from this point of view that our analysis is carried.)

We would like to study static, time-independent operators, so we will from now on drop the negative definite part of the metric, thus restricting our attention to spacelike submanifolds where $t$ is constant. Such an approach is meaningful, at least in principle, since a wave function defined on a Cauchy surface of a static spacetime (as is the case with the exterior of Schwarzschild) maintains the classical probabilistic one-particle interpretation. The spacelike submanifolds are all isometric to the Riemannian manifold $(\mathcal{E},\g)$, where $\mathcal{E} = \{ x \in \mathbb{R}^n : |x|>1 \}$ as before and 
\begin{equation}
\g=\frac{r^{n-2}}{r^{n-2}-1} dr \otimes dr + r^2 g_{S^{n-1}}
\end{equation}
is the $n$-dimensional reduced Schwarzschild metric. Note that these would no longer be spacelike (or Cauchy) if extended to include any part of the interior region.

The Riemannian gradient of a differentiable function $f:\mathcal{E} \rightarrow \mathbb{R}$ in this case reads 
\begin{equation}
\grad_\g f = \frac{r^{n-2}-1}{r^{n-2}} \frac{\partial f}{\partial r} \partial_r + \frac{1}{r^2} \grad_{S^{n-1}}f,
\end{equation}
and the associated volume form in polar coordinates is given by 
\begin{equation}
\omega_\g = r^{n-1} \sqrt{\frac{r^{n-2}}{r^{n-2}-1}} dr \wedge \omega_{S^{n-1}},
\end{equation}
so, in particular, the Riemannian measure $v_\g$ is related to the standard Lebesgue measure $dx$ by 
\begin{equation}
dv_\g = \sqrt{\frac{r^{n-2}}{r^{n-2}-1}}dx.
\end{equation}
The subscript $\g$ will be dropped from now when the context is clear.

The Riemannian distance from the event horizon is uniquely determined as the radial function $d:\mathcal{E} \rightarrow (0,\infty)$ such that $|\grad_\g d|=1$ and $\lim_{|x| \rightarrow 1} d(x)=0$. This function is given by 
\begin{equation}
d = \int_1^r \sqrt{\frac{\xi^{n-2}}{\xi^{n-2}-1}} d\xi.
\end{equation}
The integral is convergent for all $n \geq 3$ and can be calculated explicitly in the lowest dimensions $n=3,4$. In particular, we have 
\begin{align*}
&d=\sqrt{r} \sqrt{r-1} + \log (\sqrt{r}+\sqrt{r-1})\quad \textrm{for } n=3, \\
&d=\sqrt{r^2-1}\quad \textrm{for } n=4.
\end{align*}
For higher dimensions, an expression can always be obtained via series expansion, however this is impractical for our purposes so we spare the details. Note, also, that $d$ is a decreasing function of $n$. In Figure 1 we give the plot of $d$ for the first three dimensions of interest. In addition, by elementary calculations it easy to show that for $r$ close to $1$ we have the approximation 
\begin{align*}
d = \frac{2}{\sqrt{n-2}} \sqrt{r-1} +o(\sqrt{r-1}) \quad \text{as } r\rightarrow 1,
\end{align*}
and it is obvious that $d = r+o(r)$ for large $r$.

\begin{figure}
\centering
\includegraphics[scale=0.8]{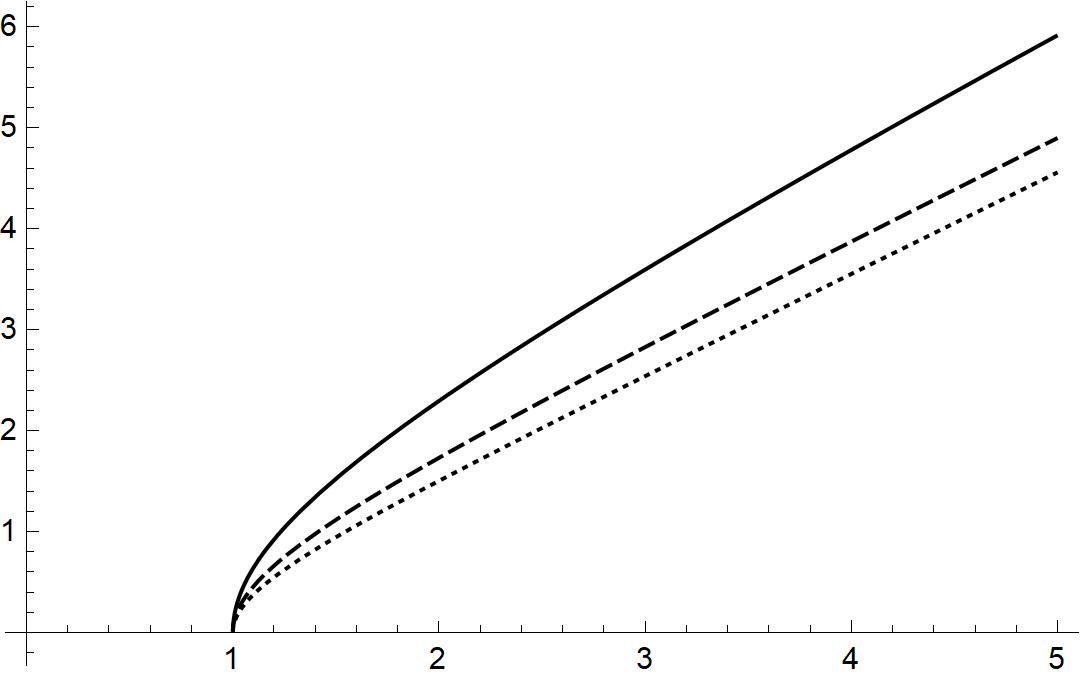}
\caption{The Riemannian distance $d$ of a point from the event horizon as a function of $r$ for dimensions $n=3,4,5$.}
\end{figure}

\section{Hardy inequality}

Our goal is to establish a critical potential $V$ that is singular on the event horizon such that the inequality 
\begin{equation}
\int_\mathcal{E} |\grad \psi|^2 dv \geq \int_\mathcal{E} V |\psi|^2 dv
\end{equation}
holds for a set of admissible functions $\psi$ containing $C^1_c(\mathcal{E})$. To have a better comparison of this potential to the one that appears in the standard Hardy inequality, we choose to express it in the form $V= \kappa / \delta^2$ where $\delta$ is an appropriate ``distance'' and $\kappa$ is a positive constant.

For a radial function $\psi = \psi (r)$, such an inequality essentially takes the form 
\begin{equation}\label{1-dim}
\int_1^\infty |\psi'(r)|^2 r^{n-1} \sqrt{\frac{r^{n-2}-1}{r^{n-2}}}dr \geq \kappa \int_1^\infty \frac{|\psi(r)|^2}{\delta^2(r)} r^{n-1} \sqrt{\frac{r^{n-2}}{r^{n-2}-1}}dr.
\end{equation}
Note that once we specify a critical potential for the radial case, the same potential will also be critical for the general case as well, due to the spherical symmetry of the problem.

In what follows, we require the following lemma.

\begin{lemma}
The general solution of the differential equation 
\begin{equation}\label{ODE1}
f'(x) - \frac{n-1}{x} f(x) = \pm \sqrt{\frac{x^{n-2}}{x^{n-2}-1}}
\end{equation}
on $(1,\infty)$ is given by 
\begin{equation}
f(x) = \frac{2}{n-2} \bigg(\lambda \pm \sqrt{\frac{x^{n-2}-1}{x^{n-2}}} \bigg) x^{n-1}
\end{equation}
for each sign, respectively.
\end{lemma}

\begin{proof}
It is straightforward to verify that these are indeed solutions. Other than that, for each choice of the sign, \eqref{ODE1} is an ordinary differential equation with $C^1$ coefficients in $(1,\infty)$, so its solution is unique up to the parameter $\lambda$ that is subject to initial conditions.
\end{proof}

The reason why we are concerned with this specific ODE is because the critical potential for the black hole is given in terms of its solutions. In particular, we will prove a Hardy inequality for the function 

\begin{equation}
\delta(r)= \left\{
\begin{array}{ll}
      2r^{n-1}\sqrt{\frac{r^{n-2}-1}{r^{n-2}}} & 1<r<(4/3)^{\frac{1}{n-2}} \\
      2r^{n-1}\Big( 1-\sqrt{\frac{r^{n-2}-1}{r^{n-2}}} \Big) & r \geq (4/3)^{\frac{1}{n-2}}\\
\end{array} 
\right. .
\end{equation}

The point $R=(4/3)^{\frac{1}{n-2}}$ is the point where the two branches meet, so the function is continuous everywhere and smooth a.e. (with the single exception of $r=R$, where the derivative is discontinuous). In Figure 2 we give a plot of $\delta$ for the first three dimensions of interest. There are also some additional properties to consider. The function $f(r)=(n-2)^{-1} \delta (r)$ consists of two branches that are solutions to the ODE \eqref{ODE1} for each sign, respectively. The choice $\lambda = 0$ on the first branch was necessary to ensure that $\delta(1)=0$. The choice $\lambda=1$ on the second branch is also necessary if we require the linear asymptotic behaviour $\delta(r) = r +o(r)$ for large $r$. The second requirement is not strictly necessary, but since the Schwarzschild metric is asymptotically flat, we would like to consider potentials that are asymptotically comparable to the Euclidean one.

\begin{figure}
\centering
\includegraphics[scale=0.8]{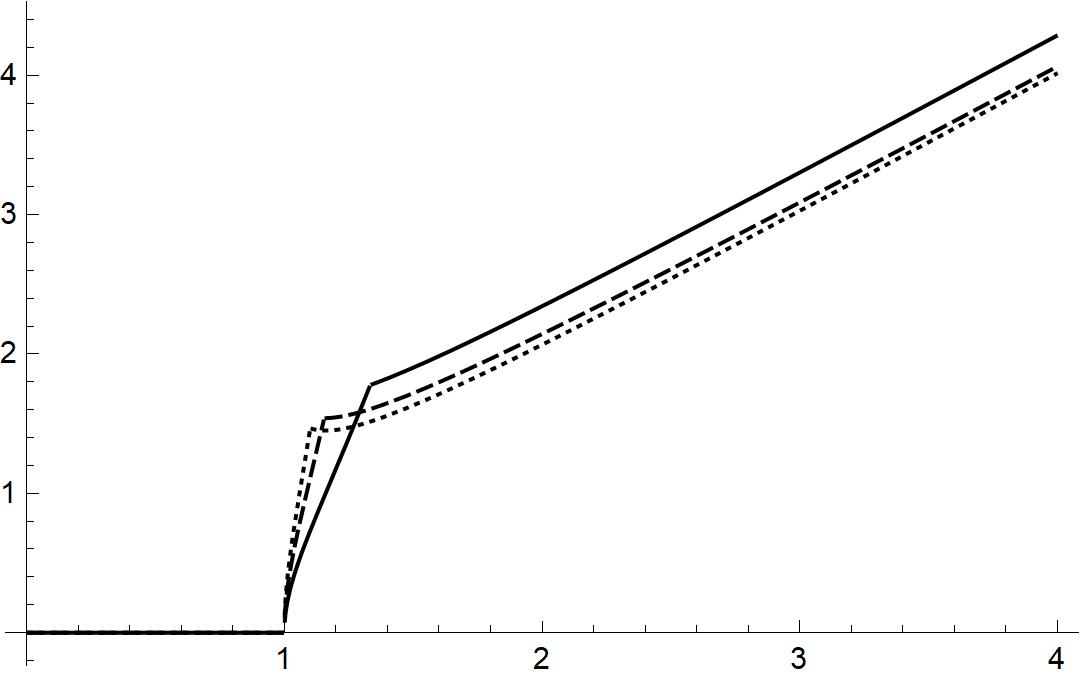}
\caption{The critical distance $\delta$ as a function of $r$ for dimensions $n=3,4,5$.}
\end{figure}

We prove the following.

\begin{theorem}[Hardy Inequality in the presence of a Schwarzschild Black Hole]\label{delta-hardy}
Let $n \geq 3$, and let $\psi: \mathcal{E} \rightarrow \mathbb{R}$ be locally absolutely continuous\footnote{This means: absolutely continuous when restricted to compact straight line segments. It is necessary and sufficient for integration by parts.} and satisfy the limit conditions $\psi(x)(|x|-1)^{-1/4} \rightarrow 0$ as $|x|\rightarrow 1$ and $\psi(x)|x|^{\frac{n-2}{2}} \rightarrow 0$ as $|x|\rightarrow \infty$. Then the inequality 
\begin{equation}\label{hardy-ineq-thm}
\int_\mathcal{E} |\grad \psi|^2 dv \geq \bigg( \frac{n-2}{2} \bigg)^2 \int_\mathcal{E} \frac{\psi^2}{\delta^2} dv
\end{equation}
is valid. Moreover, the constant $\big( \frac{n-2}{2} \big)^2$ is sharp.
\end{theorem}

\begin{proof}
Let $\psi = \psi(r,\omega)$ in polar coordinates, and let $f(r)=(n-2)^{-1}\delta(r)$ as before. We denote the first branch of $f$ by $f_+$ and the second branch by $f_-$. Then we have 
\begin{align*}
f_\pm'(r) - \frac{n-1}{r} f_\pm(r) = \pm \sqrt{\frac{r^{n-2}}{r^{n-2}-1}},
\end{align*}
and 
\begin{align*}
&\int_1^\infty \frac{|\psi(r,\omega)|^2}{f^2(r)} r^{n-1} \sqrt{\frac{r^{n-2}}{r^{n-2}-1}}dr \\
&\hspace{1cm} =\int_1^R \frac{|\psi(r,\omega)|^2}{f_+^2(r)} r^{n-1} \sqrt{\frac{r^{n-2}}{r^{n-2}-1}}dr + \int_R^\infty \frac{|\psi(r,\omega)|^2}{f_-^2(r)} r^{n-1} \sqrt{\frac{r^{n-2}}{r^{n-2}-1}}dr.
\end{align*}
By our assumptions, it follows that 
\begin{align*}
\frac{r^{n-1}}{f_\pm^2(r)} \sqrt{\frac{r^{n-2}}{r^{n-2}-1}} = \pm \frac{r^{n-1}}{f_\pm^2(r)} \bigg( f_\pm'(r) - \frac{n-1}{r} f_\pm(r) \bigg) = \mp \frac{\partial}{\partial r} \bigg( \frac{r^{n-1}}{f_\pm(r)} \bigg),
\end{align*}
and direct substitution yields 
\begin{align*}
&\int_1^\infty \frac{|\psi(r,\omega)|^2}{f^2(r)} r^{n-1} \sqrt{\frac{r^{n-2}}{r^{n-2}-1}}dr \\
&\hspace{1cm} = -\int_1^R |\psi(r,\omega)|^2 \frac{\partial}{\partial r} \bigg( \frac{r^{n-1}}{f_+(r)} \bigg)dr + \int_R^\infty |\psi(r,\omega)|^2 \frac{\partial}{\partial r} \bigg( \frac{r^{n-1}}{f_-(r)} \bigg)dr.
\end{align*}
Performing integration by parts and taking into account the limit conditions, this is equal to 
\begin{align*}
&-2|\psi(R,\omega)|^2\frac{R^{n-1}}{f(R)} + 2 \int_1^R \psi(r,\omega) \frac{\partial \psi}{\partial r}(r,\omega) \frac{r^{n-1}}{f(r)}dr \\
&\hspace{1cm }- 2 \int_R^\infty \psi(r,\omega) \frac{\partial \psi}{\partial r}(r,\omega) \frac{r^{n-1}}{f(r)}dr \leq 2 \int_1^\infty \frac{|\psi(r,\omega)|}{f(r)}\bigg| \frac{\partial \psi}{\partial r}(r,\omega) \bigg| r^{n-1}dr.
\end{align*}
Application of the Cauchy-Schwarz inequality then yields 
\begin{align*}
&\int_1^\infty \frac{|\psi(r,\omega)|^2}{f^2(r)} r^{n-1} \sqrt{\frac{r^{n-2}}{r^{n-2}-1}}dr \leq 2 \bigg( \int_1^\infty \frac{|\psi(r,\omega)|^2}{f^2(r)} r^{n-1} \sqrt{\frac{r^{n-2}}{r^{n-2}-1}}dr \bigg)^{1/2} \\
&\hspace{1cm}\times \bigg( \int_1^\infty \bigg| \frac{\partial \psi}{\partial r}(r,\omega) \bigg|^2 r^{n-1} \sqrt{\frac{r^{n-2}-1}{r^{n-2}}}dr\bigg)^{1/2}.
\end{align*}
To finish the proof of the inequality, it suffices to integrate both sides over $S^{n-1}$ and apply the Cauchy-Schwarz inequality one more time. This proves the inequality.

As for sharpness, arguing as in \cite{P} one can see that the one-dimensional inequality \eqref{1-dim} is sharp for $\kappa = \big( \frac{n-2}{2} \big)^2$, as it is essentially a suitable reparametrisation of the classic Hardy inequality with two endpoints. Then a minimising sequence of that inequality can be extended via spherical symmetry to a minimising sequence of \eqref{hardy-ineq-thm}.
\end{proof}

\paragraph{\textbf{Comparison with the Riemannian distance}} It turns out that the Riemannian distance $d$ and the function $\delta$ that appears in Theorem \ref{delta-hardy} have the same type of asymptotic behavior as $r \rightarrow 1$ and $r \rightarrow \infty$. This observation, combined with the fact that both functions are continuous implies that the quantity $d(r)/\delta(r)$ must be bounded in $(1,\infty)$ above and below by positive contants. Consequently, we get an estimate 
\begin{equation}\label{d-est}
\int_\mathcal{E} |\grad \psi|^2 dv \geq \bigg( \frac{n-2}{2} \bigg)^2 \int_\mathcal{E} \frac{\psi^2}{\delta^2} dv \geq \bigg( \frac{n-2}{2} \bigg)^2 \inf_{(1,\infty)} \bigg( \frac{d^2}{\delta^2} \bigg) \int_\mathcal{E} \frac{\psi^2}{d^2} dv.
\end{equation}
In particular, we prove the following.

\begin{corollary}[Existence of Riemannian Hardy constant]\label{d-hardy}
Let $n\geq 3$. There exists a constant $\kappa=\kappa(n)>0$ such that for all $\psi: \mathcal{E} \rightarrow \mathbb{R}$ as in Theorem \ref{delta-hardy} there holds 
\begin{equation}
\int_\mathcal{E} |\grad \psi|^2 dv \geq \kappa \int_\mathcal{E} \frac{\psi^2}{d^2} dv.
\end{equation}
In particular, we have the continuous embedding
\begin{equation}
H^1_0(\mathcal{E}) \hookrightarrow L^2(\mathcal{E},d^{-2}).
\end{equation}
\end{corollary}

\begin{proof}
In view of \eqref{d-est}, it suffices to show that $\inf (d/\delta) >0$. The function $d/\delta$ is continuous and positive in $(1,\infty)$, so it remains to check the limits as $r\rightarrow 1$ and $r\rightarrow \infty$. By elementary methods, one can show that
\begin{equation}
    d(r) = 
    \left\{
        \begin{array}{ll}
            \frac{2}{\sqrt{n-2}}\sqrt{r-1} +o(\sqrt{r-1}) & \text{ as } r\rightarrow 1 \\
            r+o(r) & \text{ as } r\rightarrow\infty
        \end{array}
    \right.
\end{equation}
while
\begin{equation}
    \delta(r) = 
    \left\{
        \begin{array}{ll}
            2\sqrt{n-2}\sqrt{r-1} +o(\sqrt{r-1}) & \text{ as } r\rightarrow 1 \\
            r+o(r) & \text{ as } r\rightarrow\infty
        \end{array}
    \right.
\end{equation}
This proves that 
\begin{align*}
\lim_{r\rightarrow 1} d/\delta = \frac{1}{n-2},\quad \lim_{r\rightarrow \infty} d/\delta = 1,
\end{align*}
which are both positive. This completes the proof.
\end{proof}

\begin{remark}
Our method is sufficient to demonstrate the existence of a positive Hardy constant, and also provides a way - at least in principle - to calculate lower bounds for the best constant, provided that one can find the minimum of $d/\delta$. The extremal analysis of $d/\delta$ proves to be quite a difficult task, however, and will not be pursued further in the present work. One may convince oneself, at least in the case $n=3$, by looking at a plot of $d/\delta$, that the minimum is achieved at $r=R=4/3$, which would imply that a lower bound for the best constant in this case is
\begin{equation}
\kappa(3)\geq \bigg( \frac{9}{32}\bigg( \frac{2}{3} + \frac{\log 3}{2} \bigg)\bigg)^2 \simeq 0.117,
\end{equation}
which is significantly smaller, although of the same order of magnitude, than the corresponding Euclidean value $1/4$. While this may (or may not) be the value of the best constant, its exact determination cannot be achieved in this manner and is, therefore, left as an open problem.
\end{remark}

\section{Heisenberg inequality}

An inequality that is closely related to the Hardy inequality is the Heisenberg inequality, better known as Heisenberg's uncertainty principle. In Euclidean space it takes the form 
\begin{equation}
\frac{n}{2} \int_{\mathbb{R}^n} |\psi(x)|^2 dx \leq \bigg( \int_{\mathbb{R}^n} |x|^2 |\psi(x)|^2 dx \bigg)^{1/2} \bigg( \int_{\mathbb{R}^n} |\grad \psi(x)|^2 dx \bigg)^{1/2}
\end{equation}
for $\psi \in C^1_c(\mathbb{R}^n)$. There have also been some generalisations in the Riemannian setting; most notably, see \cite{K}.

Here we prove an analogue for Schwarzchild geometry which involves an induced ``distance'' $s:\mathcal{E} \rightarrow \mathbb{R}$ from the event horizon, given by 
\begin{equation}
s(r)=\frac{1}{r^{n-1}} \int_1^r \sqrt{\frac{\xi^{n-2}}{\xi^{n-2}-1}} \xi^{n-1} d\xi.
\end{equation}

Before we state and prove the inequality, it is worthwhile to comment on some of the properties of $s$. First, $s(r)$ is a solution of the initial value problem 
\begin{equation}
s'(r)+\frac{n-1}{r} s(r) = \sqrt{\frac{r^{n-2}}{r^{n-2}-1}},\ \ \ s(1)=0
\end{equation}
for $1 \leq r < \infty$. Similar to the Riemannian distance $d$, $s$ has asymptotic behavior 
\begin{equation}
    s(r) = 
    \left\{
        \begin{array}{ll}
            \frac{2}{\sqrt{n-2}}\sqrt{r-1} +o(\sqrt{r-1}) & \text{ as } r\rightarrow 1 \\
            r/n+o(r) & \text{ as } r\rightarrow\infty
        \end{array}
    \right. ,
\end{equation}
which is the same for small $r$ and scaled by a factor of $1/n$ for large $r$. The induced distance can be also calculated explicitly for $n=3,4$. Specifically, we have 
\begin{align*}
&s=\frac{1}{24} \sqrt{r-1} \sqrt{r} \bigg( 8+\frac{10}{r} +\frac{15}{r^2} \bigg) + \frac{5}{8r^2} \log (\sqrt{r}+\sqrt{r-1})\quad \textrm{for } n=3, \\
&s=\frac{\sqrt{r^2-1}}{8} \bigg( 2+\frac{2}{r^2} \bigg)+\frac{3}{8r^3} \log (r+\sqrt{r^2-1})\quad \textrm{for } n=4.
\end{align*}
In Figure 3, we give a plot of $s$ for the first three dimensions of interest.

\begin{figure}
\centering
\includegraphics[scale=0.8]{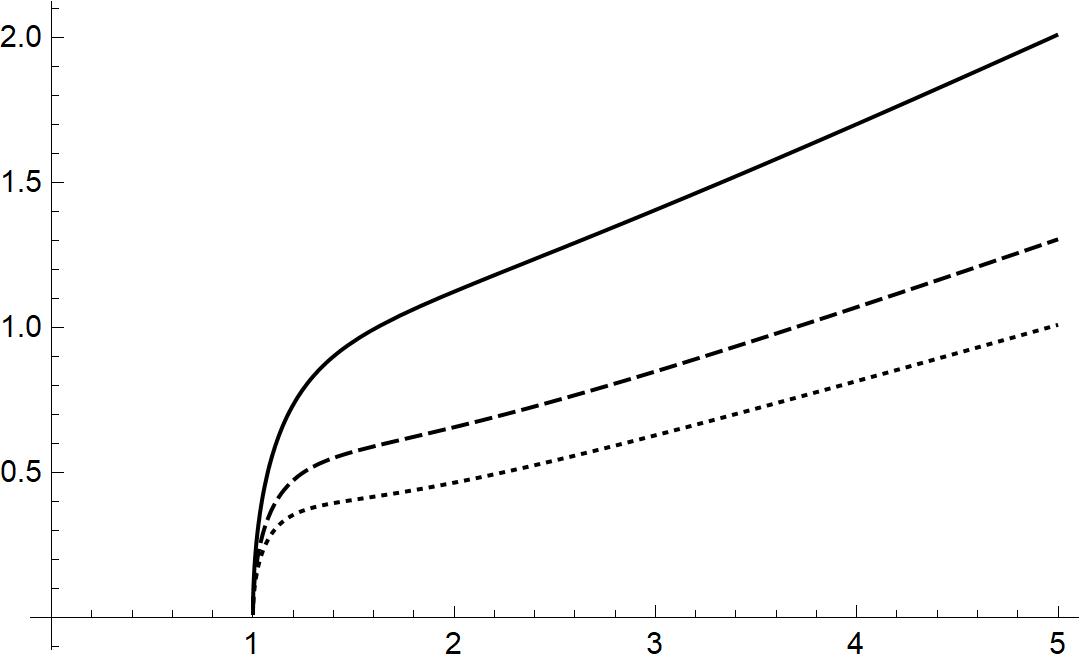}
\caption{The induced distance $s$ of a point from the event horizon as a function of $r$ for dimensions $n=3,4,5$.}
\end{figure}

\begin{theorem}[Heisenberg inequality in the presence of a Schwarzschild Black Hole]\label{Heisenberg}
Let $\psi:\mathcal{E} \rightarrow \mathbb{R}$ be locally absolutely continuous and satisfy the limit condition $\psi^2(x) |x|^n \rightarrow 0$ as $|x| \rightarrow \infty$. Then 
\begin{equation}
\frac{1}{2} \int_\mathcal{E} |\psi|^2 dv \leq \bigg( \int_\mathcal{E} s^2 |\psi|^2 dv \bigg)^{1/2} \bigg( \int_\mathcal{E} |\grad \psi|^2 dv \bigg)^{1/2}.
\end{equation}
The constant $1/2$ is sharp and attainable in that class of functions by any radial function of the form 
\begin{equation}
A\exp \bigg( -B\int_1^r s(\xi) \sqrt{\frac{\xi^{n-2}}{\xi^{n-2}-1}} d\xi \bigg),\quad A\in\R,B>0.
\end{equation}
\end{theorem}

\begin{proof}
By integration by parts and the Cauchy-Schwarz inequality we have 
\begin{align*}
&\int_1^\infty |\psi(r,\omega)|^2 r^{n-1} \sqrt{\frac{r^{n-2}}{r^{n-2}-1}} dr \\
&\hspace{1cm}= \int_1^\infty |\psi(r,\omega)|^2 (s(r)r^{n-1})'dr \\
&\hspace{2cm}= [|\psi(r,\omega)|^2 s(r)r^{n-1} ]_1^\infty -2\int_1^\infty \psi(r,\omega) \frac{\partial \psi}{\partial r}(r,\omega) s(r) r^{n-1} dr \\
&\hspace{3cm}\leq 2 \bigg( \int_1^\infty s(r)^2 |\psi(r,\omega)|^2 r^{n-1} \sqrt{\frac{r^{n-2}}{r^{n-2}-1}} dr \bigg)^{1/2} \\
&\hspace{4cm}\times \bigg( \int_1^r \big| \frac{\partial \psi}{\partial r}(r,\omega) \big|^2 r^{n-1} \sqrt{\frac{r^{n-2}-1}{r^{n-2}}} dr \bigg)^{1/2}.
\end{align*}
Integrating over $S^{n-1}$ and applying the Cauchy-Schwarz inequality one more time completes the proof of the inequality.

\medskip

\paragraph{\textit{Minimisers.}} For a radial function $\psi=\psi(r)$ to be a minimiser, there are three necessary and sufficient conditions, all arising by the requirement that we have equality in each step of the above calculations. First, in order to eliminate the boundary term, it must be that 
\begin{equation}\label{limit-condition-1}
\lim_{r \rightarrow 1} |\psi(r)|^2 s(r) r^{n-1} = 0.
\end{equation}

Second, in order to have equality in
\begin{align*}
-2\int_1^\infty \psi(r) \psi'(r) s(r) r^{n-1} dr \leq 2 \int_1^\infty |\psi(r)| | \psi'(r) | s(r) r^{n-1} dr,
\end{align*}
we must have the sign condition $\psi(r)\psi'(r)\leq 0$.

Finally, to have equality in the Cauchy-Schwarz inequality, there must be a constant $B\in\R$ such that 
\begin{align*}
\psi'(r)=B\psi(r)s(r)\sqrt{\frac{r^{n-2}}{r^{n-2}-1}},
\end{align*}
and thus 
\begin{align*}
(\log|\psi(r)|)'=B s(r) \sqrt{\frac{r^{n-2}}{r^{n-2}-1}}.
\end{align*}
Integrating both sides from $1$ to $r$ yields 
\begin{align*}
\psi(r)=A\exp \bigg( B\int_1^r s(\xi) \sqrt{\frac{\xi^{n-2}}{\xi^{n-2}-1}} d\xi \bigg),
\end{align*}
where $A = \psi(1) \in\R$. The limit condition \eqref{limit-condition-1} is readily satisfied. The limit condition at infinity holds if and only if $B<0$, in which case the sign condition is also satisfied. It follows that functions of the form 
\begin{align*}
A\exp \bigg( -B\int_1^r s(\xi) \sqrt{\frac{\xi^{n-2}}{\xi^{n-2}-1}} d\xi \bigg),\quad A\in\R, B>0
\end{align*}
are indeed the minimisers of the inequality. This completes the proof.
\end{proof}

\begin{remark}
Note that, in contrast to the Euclidean case, the best constant here is $1/2$ and not $n/2$. However, the inequality is directly comparable to the Euclidean one: unless $r$ is comparable to $1$, the induced distance $s(r)$ quickly assumes the linear behaviour $s(r) \approx r/n$, which accounts for the missing $n$ in the constant. In other words, if one is not close to the event horizon, the inequality is essentially the same as the Euclidean one. Since the Schwarzschild metric is asymptotically flat, this is to be expected anyway. However, Figure 3 suggests that one need not go very far for this to happen. For $r\geq 2$ the approximation $s(r)\approx r/n$ seems to be quite good.
\end{remark}

\medskip

\paragraph{\textbf{Comparison with the Riemannian distance}} It is easy to see from the definitions that $s(r)\leq d(r)$, which immediately proves the following related Heisenberg inequality.

\begin{corollary}
Let $\psi:\mathcal{E} \rightarrow \mathbb{R}$ be as in Theorem \ref{Heisenberg}. Then 
\begin{equation}
\frac{1}{2} \int_\mathcal{E} |\psi|^2 dv \leq \bigg( \int_\mathcal{E} d^2 |\psi|^2 dv \bigg)^{1/2} \bigg( \int_\mathcal{E} |\grad \psi|^2 dv \bigg)^{1/2}.
\end{equation}
\end{corollary}

\begin{remark}
Since $s=d/n+o(d)$ as $d\rightarrow\infty$, it may well be the case that the constant $1/2$ is \textit{not} optimal here. This cannot be proven or disproven with the method we have developed so far, and is again left as an open problem.

In all, our method, in both the Hardy and the Heisenberg inequalities, has exploited the particular features of the weights that appear due to the geometry to derive potentials that particularly fit these weights so that no information is ``lost'' in the process. In this way we have established sharp inequalities for those weights. However, information may be lost when one crudely compares these weights with the Riemannian distance using $L^\infty$ estimates, and it appears that different methods are required to treat the optimality problem in those cases.
\end{remark}

\medskip

\paragraph{\textbf{Acknowledgment}} The author would like to thank prof. G. Barbatis for reviewing the article and making useful suggestions.

\bibliographystyle{amsplain}

\end{document}